\documentclass{amsart}
\usepackage[foot]{amsaddr}
\usepackage{amssymb,amsmath,amsfonts,mathptmx,cite,mathrsfs,url}
\usepackage[mathscr]{eucal}

\DeclareSymbolFont{rsfscript}{OMS}{rsfs}{m}{n}
\DeclareSymbolFontAlphabet{\mathrsfs}{rsfscript}

\usepackage{gastex}

\newtheorem{theorem}{Theorem}[section]

\newtheorem{proposition}[theorem]{Proposition}
\newtheorem{lemma}[theorem]{Lemma}
\newtheorem{corollary}[theorem]{Corollary}

\theoremstyle{remark}
\newtheorem{remark}{Remark}

\newcommand{\ais}{ai-semi\-ring}
\def\softd{{\leavevmode\setbox1=\hbox{d}%
    \hbox to 1.05\wd1{d\kern-0.4ex{\char039}\hss}}}

\numberwithin{equation}{section}

\makeatletter

\renewcommand*\subjclass[2][2010]{\def\@subjclass{#2}\@ifundefined{subjclassname@#1}{\ClassWarning{\@classname}{Unknown edition (#1) of Mathematics Subject Classification; using '2010'.}}{\@xp\let\@xp\subjclassname\csname subjclassname@#1\endcsname}}

\renewcommand{\subjclassname}{\textup{2010} Mathematics Subject Classification}

\makeatother

\begin{document}

\title[Semiring and involution identities of power groups]{Semiring and involution identities of power groups}
\thanks{Supported by the Russian Science Foundation (grant No. 22-21-00650).}

\author[S. V. Gusev]{Sergey V. Gusev}
\address{Institute of Natural Sciences and Mathematics\\
Ural Federal University\\ 620000 Ekaterinburg, Russia}
\email{sergey.gusb@gmail.com}
\email{m.v.volkov@urfu.ru}

\author[M. V. Volkov]{Mikhail V. Volkov}

\begin{abstract}
For every group $G$, the set $\mathcal{P}(G)$ of its subsets forms a semiring under set-theoretical union $\cup$ and element-wise multiplication $\cdot$ and forms an involution semigroup under $\cdot$ and element-wise inversion ${}^{-1}$. We show that if the group $G$ is finite, non-Dedekind, and solvable, neither the semiring $(\mathcal{P}(G),\cup,\cdot)$ nor the involution semigroup $(\mathcal{P}(G),\cdot,{}^{-1})$ admits a finite identity basis. We also solve the finite basis problem for the semiring of Hall relations over any finite set.

\end{abstract}



\maketitle

\section{Introduction}
\label{sec:introduction}

We assume the reader's acquaintance with the concepts of an algebra of type $\tau$ and an identity of type $\tau$; see, e.g., \cite{BuSa81}. We use expressions like $A:=B$ to emphasize that $A$ is defined to be $B$.

An \emph{additively idempotent semiring} (\ais, for short) is an algebra $\mathcal{S}=(S, +, \cdot)$ of type $(2,2)$ whose operations are associative and fulfil the distributive laws
\[
x(y+z)=xy+xz\ \text{ and }\ (y+z)x=yx+zx,
\]
and besides that, addition is commutative and fulfils the idempotency law $x+x=x$. We refer to the semigroup $(S,\cdot)$ as to the \emph{multiplicative reduct} of $\mathcal{S}$.

Examples of \ais{}s are plentiful and include many objects of importance for computer science, idempotent analysis, tropical geometry, and algebra such as, e.g., semirings of binary relations \cite{Jip17}, syntactic semirings of languages \cite{Polak01,Polak03}, tropical semirings \cite{Pin98}, endomorphism semirings of semilattices \cite{JKM09}. A natural family of \ais{}s comes from the powerset construction applied to an arbitrary semigroup. Namely, for any semigroup $(S,\cdot)$, one can multiply its subsets element-wise: for any $A,B\subseteq S$, put
\[
A\cdot B:=\{ab\mid a\in A,\ b\in B\}.
\]
It is known and easy to verify that this multiplication distributes over the set-theoretical union of subsets of $S$. Thus, denoting the powerset of $S$ by $\mathcal{P}(S)$, we see that $(\mathcal{P}(S),\cup,\cdot)$ is an \ais, called the \emph{power semiring} of $(S,\cdot)$. In semigroup theory, there exists a vast literature on \emph{power semigroups}, i.e., the multiplicative reducts of power semirings. The interest in power semigroups is mainly motivated by their applications in the algebraic theory of regular languages; see \cite[Chapter 11]{Almeida:95} and references therein. Power semirings have attracted less attention so far, but we think they deserve study no less than power semigroups as $(\mathcal{P}(S),\cup,\cdot)$ carries more information than its multiplicative reduct $(\mathcal{P}(S),\cdot)$. (For instance, it is known that in general, the power semigroup $(\mathcal{P}(S),\cdot)$ does not determine its `parent' semigroup $(S,\cdot)$ up to isomorphism \cite{Mog72,Mog73} while it is easy to see that the power semiring $(\mathcal{P}(S),\cup,\cdot)$ does.)

The power semigroups of finite groups have quite a specific structure discussed in detail in \cite[Section 11.1]{Almeida:95}. In particular, they satisfy the following implications:
\begin{gather}
ef=e^2=e\mathrel{\&}fe=f^2=f\to e=f,\label{eq:bg1}\\
ef=f^2=f\mathrel{\&}fe=e^2=e\to e=f.\label{eq:bg2}
\end{gather}
Semigroups satisfying \eqref{eq:bg1} and \eqref{eq:bg2} are called \emph{block-groups}. We refer the reader to Pin's survey~\cite{Pin:1995} for an explanation of the name `block-group' and an overview of the remarkable role played by finite block-groups in the theory of regular languages. In this paper, we deal with power semirings of finite groups, which we call \emph{power groups} for brevity, and more broadly, with \ais{}s whose multiplicative reducts are block-groups. We address the finite axiomatizability question (aka the Finite Basis Problem) for semiring identities satisfied by such \ais{}s. For power groups, we prove the following.

\begin{theorem}
\label{thm:psg}
For any finite non-abelian solvable group $(G,\cdot)$, the identities of the power group $(\mathcal{P}(G),\cup,\cdot)$ admit no finite basis.
\end{theorem}

In 2009, Dolinka \cite[Problem 6.5]{Dolinka09a} asked for a description of finite groups whose power groups have no finite identity basis (provided such finite groups exist, which was unknown at that time). Theorem~\ref{thm:psg} gives a partial answer to Dolinka's question.

Theorem~\ref{thm:psg} is a consequence of quite a general condition ensuring the absence of a finite identity basis for \ais{}s whose multiplicative reducts are block-groups with solvable subgroups (see Theorem~\ref{thm:main} below). Yet another application of this condition, combined with a result from \cite{JackRenZhao}, gives a complete solution to the Finite Basis Problem for a natural family of semirings of binary relations. Recall that a \emph{Hall relation} on a finite set $X$ is a binary relation that contains all pairs $(x,x\pi)$, $x\in X$, for a permutation $\pi\colon X\to X$. Clearly, the set $H(X)$ of all Hall relations on $X$ is closed under unions and products of relations, and hence $(H(X),\cup,\cdot)$ is a subsemiring of the $(\cup,\cdot)$-semiring of all binary relations on $X$.

\begin{theorem}
\label{thm:hall}
The identities of the \ais\ of all Hall relations on a finite set $X$ admit a finite basis if and only if $|X|=1$.
\end{theorem}

Our techniques apply equally well to yet another algebra carried by the powerset of a group. An \emph{involution semigroup} is an algebra of type $(2,1)$ whose binary operation is associative and whose unary operation $x\mapsto x^*$ fulfils the laws
\[
(xy)^*=y^*x^*\ \text{ and }\ (x^*)^*=x.
\]
If $(G,\cdot)$ is a group, each $g\in G$ has the inverse $g^{-1}$ and the map $g\mapsto g^{-1}$ naturally extends to the powerset of $G$: for each $A\in\mathcal{P}(G)$, let $A^{-1}:=\{g^{-1}\mid g\in A\}$. Clearly, $(\mathcal{P}(G),\cdot,{}^{-1})$ forms an involution semigroup. The Finite Basis Problem for involution semigroups is relatively well studied (see, e.g., \cite{Dolinka10,ADV12,ADV12a,Lee17}), but involution power semigroups of groups do not fall in any previously studied class. Our next result applies to involution power semigroups of finite solvable groups containing at least one non-normal subgroup. Recall that a group all of whose subgroups are normal is called a \emph{Dedekind} group.

\begin{theorem}
\label{thm:inv-psg}
For any finite non-Dedekind solvable group $(G,\cdot)$, the identities of the involution semigroup $(\mathcal{P}(G),\cdot,{}^{-1})$ admit no finite basis.
\end{theorem}

The paper is structured as follows. Sect.~\ref{sec:prelim} collects the necessary properties of block-groups. The technical core of the paper is Sect.~\ref{sec:identities} where we show that every finite block-group with solvable subgroups satisfies certain semigroup identities from a family of identities introduced in our earlier paper~\cite{GV22}. This allows us in Sect.~\ref{sec:proofs} to quickly deduce Theorems~\ref{thm:psg} and~\ref{thm:hall} from the main result of~\cite{GV22}. In Sect.~\ref{sec:involution} we prove an involution semigroup variant of the main result of~\cite{GV22}, which allows us to infer Theorem~\ref{thm:inv-psg}.

\section{Preliminaries}
\label{sec:prelim}

The present paper uses a few basic notions of semigroup theory such as idempotents, ideals, and Rees quotients. They all can be found in the early chapters of any general semigroup theory text, e.g., \cite{Clifford&Preston:1961,Howie:1995}. We also assume the reader's acquaintance with rudiments of group theory, including the concept of a solvable group, see, e.g., \cite{Hall:1959}.

We recall the definition of a Brandt semigroup as we will need some calculations in such semigroups. Let $I$ be a non-empty set, $\mathcal{G}=(G,\cdot)$ a group, and $0\notin G$ a fresh symbol. The \emph{Brandt semigroup} $\mathcal{B}_{G,I}$ over $\mathcal{G}$ has the set $B_{G,I}:=I \times G \times I\cup \{0\}$ as its carrier, and the multiplication in $\mathcal{B}_{G,I}$ is defined by
\begin{align*}
&(\ell_1,g_1,r_1)\cdot (\ell_2,g_2,r_2):=
\begin{cases}
(\ell_1,g_1g_2,r_2)&\text{if }r_1=\ell_2,\\
0&\text{otherwise,}
\end{cases}&&\text{for all $\ell_1,\ell_2,r_1,r_2\in I,\ g_1,g_2\in G$,}\\
&(\ell,g,r)\cdot 0=0\cdot(\ell,g,r)=0\cdot0:=0&&\text{for all $\ell,r\in I,\ g\in G$.}
\end{align*}
We register a property of Brandt semigroups that readily follows from the definition.

\begin{lemma}
\label{lem:brandt}
Let $I$ be a non-empty set, let $\mathcal{G}=(G,\cdot)$ be a group, and for any $\ell,r\in I$, let $G_{\ell r}:=\{(\ell,g,r)\mid g\in G\}$. Then $(G_{\ell r},\cdot)$ is a maximal subgroup of the Brandt semigroup $\mathcal{B}_{G,I}$ if $\ell=r$; otherwise the product of any two elements from $G_{\ell r}$ is\/ $0$.
\end{lemma}

An element $a$ of a semigroup $(S,\cdot)$ is said to be \emph{regular} if there exists $b\in S$ satisfying $aba=a$ and $bab=b$; any such $b$ is called an \emph{inverse of $a$}. A semigroup is called \emph{regular}  [resp., \emph{inverse}] if every its element has an inverse [resp., a unique inverse]. Every group $\mathcal{G}=(G,\cdot)$ is an inverse semigroup, and the unique inverse of an element $g\in G$ is nothing but its group inverse $g^{-1}$. The Brandt semigroup $\mathcal{B}_{G,I}$ over $\mathcal{G}$ is inverse as well: for $\ell,r\in I$ and $g\in G$, the unique inverse of $(\ell,g,r)$ is $(r,g^{-1},\ell)$ and the unique inverse of 0 is 0.

Now we present a few properties of block-groups needed for Sect.~\ref{sec:identities}. They all are known, but in some cases, we failed to find any reference that could be used directly (rather one has to combine several facts scattered over the literature). In such cases, we include easy direct proofs for the reader's convenience.

The definition of a block-group can be restated in terms of inverse elements, and it is this version of the definition that is often used in the literature to introduce block-groups.
\begin{lemma}
\label{lem:at_most_one_inv}
A semigroup $(B,\cdot)$ is a block-group if and only if every element in $B$ has at most one inverse.
\end{lemma}

\begin{proof}
The `if' part. Take any $e,f\in B$ that satisfy the antecedent of the implication \eqref{eq:bg1}, that is, $ef=e^2=e$ and $fe=f^2=f$. Multiplying $ef=e$ by $e$ on the right and using $e^2=e$, we get $efe=e$, and similarly,
multiplying $fe=f$ by $f$ on the right and using $f^2=f$, we get $fef=f$. Hence $f$ is an inverse of $e$. On the other hand, we have $e=e^2=e^3$ so that $e$ is an inverse of itself. Since $e$ has at most one inverse, we conclude that $e=f$ so that \eqref{eq:bg1} holds. By symmetry, the implication \eqref{eq:bg2} holds as well.

The `only if' part. Take any $a\in B$ and suppose that both $b$ and $c$ are inverses of $a$. Letting $e:=ba$ and $f:=ca$, we have $ef=baca=ba=e$, $e^2=baba=e$, and similarly, $fe=caba=ca=f$, $f^2=caca=f$. Since $(B,\cdot)$ satisfies the implication \eqref{eq:bg1}, we get $e=f$, that is, $ba=ca$. By symmetry, the implication \eqref{eq:bg2} ensures $ab=ac$. Now we have $b=bab=cab=cac=c$.
\end{proof}

Lemma~\ref{lem:at_most_one_inv} shows that regular block-groups are inverse, so that one can think of block-groups as a sort of non-regular analogs of inverse semigroups.

For any semigroup $(S,\cdot)$, we denote the set of all its idempotents by $E(S)$. A semigroup $(S,\cdot)$ is called $\mathscr{J}$-\emph{trivial} if every principal ideal of $(S,\cdot)$ has a unique generator and \emph{periodic} if every one-generated subsemigroup of $(S,\cdot)$ is finite. In \cite[Proposition 2.3]{MP84} it is proved that a finite monoid $(S,\cdot)$ is a block-group if and only if the subsemigroup $(\langle E(S)\rangle,\cdot)$ generated by $E(S)$ is $\mathscr{J}$-trivial. In fact, the proof in \cite{MP84} uses periodicity of $(S,\cdot)$ rather than its finiteness and does not use the identity element of $(S,\cdot)$. Thus, we have the following characterization of periodic block-groups.

\begin{proposition}
\label{prop:block-group}
A periodic semigroup is a block-group if and only if the subsemigroup generated by its idempotents is $\mathscr{J}$-trivial.
\end{proposition}

As a consequence, we have the following observation.

\begin{lemma}
\label{lem:ef-idempotent}
If $(S,\cdot)$ is a periodic block-group, then every regular element of the subsemigroup $(\langle E(S)\rangle,\cdot)$ lies in $E(S)$.
\end{lemma}

\begin{proof}
Let $a\in\langle E(S)\rangle$ be a regular element and $b$ its inverse. The FitzGerald Trick \cite[Lemma 1]{FitzGerald:70} shows that $b\in\langle E(S)\rangle$ as well. Since $a=aba$ and $b=bab$ generate the same ideal in $(\langle E(S)\rangle,\cdot)$, Proposition \ref{prop:block-group} implies $a=b$. Thus, $a=a^3$, whence $a$ and $a^2$ also generate the same ideal in $(\langle E(S)\rangle,\cdot)$. Now Proposition \ref{prop:block-group} implies $a=a^2\in E(S)$.
\end{proof}

\begin{lemma}
\label{lem:fd}
If a Brandt semigroup $(B,\cdot)$ occurs as an ideal in a periodic block-group $(S,\cdot)$, then for all $b\in B,\ f\in E(S)$, either $fb=b$ or $fb=0$.
\end{lemma}

\begin{proof}
If the semigroup $(B,\cdot)$ is represented as $\mathcal{B}_{G,I}$ for some group $(G,\cdot)$ and some non-empty set $I$, then each non-zero idempotent is of the form $(i,e,i)$ where $i\in I$ and $e$ is the identity element of the group $\mathcal{G}$. Take an arbitrary element $b=(\ell,g,r)\in I\times G\times I$. Then $b=(\ell,e,\ell)b$ whence $fb=f(\ell,e,\ell)b$ for each $f\in E(S)$.  The product $f(\ell,e,\ell)$ lies in $B$ and so it is regular. By Lemma \ref{lem:ef-idempotent} $f(\ell,e,\ell)$ is an idempotent. If $f(\ell,e,\ell)\ne0$, then $f(\ell,e,\ell)=(j,e,j)$ for some $j\in I$. Multiplying the equality through by $(j,e,j)$ on the right yields $f(\ell,e,\ell)(j,e,j)=(j,e,j)$ whence $j=\ell$. We conclude that either $f(\ell,e,\ell)=(\ell,e,\ell)$ or $f(\ell,e,\ell)=0$, whence either $fb=b$ or $fb=0$.
\end{proof}

A \emph{principal series} of a semigroup $(S,\cdot)$ is a chain
\begin{equation}
\label{eq:series}
S_0\subset S_1\subset \dots \subset S_h=S
\end{equation}
of ideals $S_j$ of $(S,\cdot)$ such that $S_0$ is the least ideal of $(S,\cdot)$ and there is no ideal of $(S,\cdot)$ strictly between $S_{j-1}$ and $S_j$ for $j=1,\dots,h$. The \emph{factors} of the series \eqref{eq:series} are the Rees quotients $(S_j/S_{j-1},\cdot)$, $j=1,\dots,h$. The number $h$ is called the \emph{height} of the series \eqref{eq:series}.

Recall that a semigroup $(Z,\cdot)$ with a zero element 0 is called a \emph{zero semigroup} if $ab = 0$ for all $a,b\in Z$. A semigroup $(T,\cdot)$ with a zero element 0 is 0-\emph{simple} if it is not a zero semigroup and $\{0\}$ and $T$ are the only ideals of $(T,\cdot)$. A semigroup $(T,\cdot)$ is \emph{simple} if $T$ is its only ideal. (Observe that a simple semigroup may consist of a single element while any 0-simple semigroup has at least two elements.) By \cite[Lemma 2.39]{Clifford&Preston:1961}, if \eqref{eq:series} is a principal series, then $(S_0,\cdot)$ is a simple semigroup and every factor $(S_j/S_{j-1},\cdot)$, $j=1,\dots,h$, is either a 0-simple semigroup or a zero semigroup with at least two elements.

Restricting a classical result of semigroup theory (see \cite[Theorem~3.5]{Clifford&Preston:1961} or \cite[Theorem 3.2.3]{Howie:1995}) to the case of periodic block-groups yields the following.
\begin{lemma}
\label{lem:bg_c0s}
If a periodic block-group is simple, it is a group, and if it is $0$-simple, it is a Brandt semigroup.
\end{lemma}

\begin{corollary}
\label{cor:series}
If a periodic block-group possesses a principal series, then its least ideal is a group and the non-zero factors of the series are Brandt semigroups.
\end{corollary}

We call \eqref{eq:series} a \emph{Brandt series} if $(S_0,\cdot)$ is a group and every factor $(S_j/S_{j-1},\cdot)$ with $j=1,\dots,h$, which is not a zero semigroup, is a Brandt semigroup. Thus, Corollary~\ref{cor:series} can be restated as saying that a principal series of a periodic block-group is a Brandt series (provided the series exists). The converse is also true, even without periodicity.

\begin{lemma}
\label{lem:bg_series}
If a semigroup possesses a Brandt series, it is a block-group.
\end{lemma}

\begin{proof}
Let \eqref{eq:series} be a Brandt series in a semigroup $(S,\cdot)$. Take any regular element $a\in S$ and let $j\in\{0,1,\dots,h\}$ be the least number with $a\in S_j$. Recall that each inverse of $a$ generates the same ideal of $(S,\cdot)$ as does $a$. Therefore, if $j=0$, all inverses of $a$ lie in the least ideal of $(S,\cdot)$ which is a group. Then all inverses of $a$ are equal since they coincide with the group inverse of $a$ in the group $(S_0,\cdot)$. If $j>0$, we have $a\in S_j\setminus S_{j-1}$, and all inverses of $a$ also lie in $S_j\setminus S_{j-1}$. Their images in the Rees quotient $(S_j/S_{j-1},\cdot)$ are inverses of the image of $a$, and hence, they coincide as the quotient must be an inverse semigroup. Since non-zero elements of $S_j/S_{j-1}$ are in a 1-1 correspondence with the elements of $S_j\setminus S_{j-1}$, all inverses of $a$ are equal also in this case. Now Lemma~\ref{lem:at_most_one_inv} shows that $(S,\cdot)$ is a block-group.
\end{proof}

We also need the following property of $\mathscr{J}$-trivial semigroups.

\begin{lemma}[{\!\cite[Lemma 8.2.2(iii)]{Almeida:95}}]
\label{lem:j-triv}
If a $\mathscr{J}$-trivial semigroup $(S,\cdot)$ satisfies the identity $x^p=x^{p+1}$ and $w$ is an arbitrary semigroup word, then $(S,\cdot)$ satisfies the identity $w^p=u^p$ where $u$ is the product, in any order, of the variables that occur in $w$.
\end{lemma}

\section{Identities holding in block-groups with solvable subgroups}
\label{sec:identities}

Here we aim to show that every finite block-group with solvable subgroups satisfies certain identities constructed in~\cite{GV22}. For the reader's convenience, we reproduce the construction in detail so that no acquaintance with~\cite{GV22} is necessary for understanding the results and the proofs of this section.

For any $i,h\ge1$, consider the set $X_i^{(h)}:=\{x_{i_1i_2\cdots i_h}\mid i_1,i_2,\dots,i_h\in\{1,2,\dots,i\}\}$. We fix arbitrary $n,m\ge1$ and introduce a family of words $\mathbf v_{n,m}^{(h)}$ over $X_{2n}^{(h)}$, $h=1,2,\dotsc$, by induction.

For $h=1$, let
\begin{equation}
\label{eq:v1}
\mathbf v_{n,m}^{(1)}:=x_1x_2\cdots x_{2n}\,\Bigl(x_nx_{n-1}\cdots x_1\cdot x_{n+1}x_{n+2}\cdots x_{2n}\Bigr)^{2m-1}.
\end{equation}

Assuming that $h>1$ and the word $\mathbf v_{n,m}^{(h-1)}$ over $X_{2n}^{(h-1)}$ has already been defined, we create $2n$ copies of this word over the alphabet $X_{2n}^{(h)}$ as follows. We start by taking for every $j\in\{1,2,\dots,2n\}$, the substitution $\sigma_{2n,j}^{(h)}\colon X_{2n}^{(h-1)}\to X_{2n}^{(h)}$ that appends $j$ to the indices of its arguments, that is,
\begin{equation}
\label{eq:sigma_j}
\sigma_{2n,j}^{(h)}(x_{i_1i_2\dots i_{h-1}}):= x_{i_1i_2\dots i_{h-1}j} \text{ for all } i_1,i_2,\dots,i_{h-1}\in\{1,2,\dots,2n\}.
\end{equation}
Then we let $\mathbf v_{n,m,j}^{(h-1)}:=\sigma_{2n,j}^{(h)}(\mathbf v_{n,m}^{(h-1)})$ and define
\begin{equation}
\label{eq:vnh}
\mathbf v_{n,m}^{(h)}:=\mathbf v_{n,m,1}^{(h-1)}\cdots\mathbf v_{n,m,2n}^{(h-1)}\Bigl(\mathbf v_{n,m,n}^{(h-1)}\cdots\mathbf v_{n,m,1}^{(h-1)}\cdot \mathbf v_{n,m,n+1}^{(h-1)}\cdots\mathbf v_{n,m,2n}^{(h-1)}\Bigr)^{2m-1}.
\end{equation}

Comparing the definitions \eqref{eq:v1} and \eqref{eq:vnh}, one readily sees that the word $\mathbf v_{n,m}^{(h)}$ is nothing but the image of $\mathbf v_{n,m}^{(1)}$ under the substitution $x_j \mapsto \mathbf v_{n,m,j}^{(h-1)}$, $j\in\{1,2,\dots,2n\}$.

\begin{remark}
\label{rm:vnh}
It is easy to see that the word $\mathbf v_{n,m}^{(h+r)}$ is the image of the word $\mathbf v_{n,m}^{(h)}$ under a substitution $\zeta$ that sends every variable from $X^{(h)}_{2n}$ to a word obtained from $\mathbf v_{n,m}^{(r)}$ by renaming its variables. Indeed, in terms of the substitutions $\sigma_{2n,j}^{(h)}$ from \eqref{eq:sigma_j}, one can express $\zeta$ as follows:
\[
\zeta(x_{i_1i_2\dots i_h}):=\sigma_{2n,i_h}^{(r+h)}(\cdots(\sigma_{2n,i_2}^{(r+2)}(\sigma_{2n,i_1}^{(r+1)}(\mathbf v_{n,m}^{(r)})))\cdots)\ \text{ for all } i_1,i_2,\dots, i_h\in\{1,2,\dots,2n\}.
\]
Simply put, $\zeta(x_{i_1i_2\dots i_h})$ appends $i_1i_2\dots i_h$ to the indices of all variables of $\mathbf v_{n,m}^{(r)}$.
\end{remark}

For any $n,k\ge 0$ with $n+k>0$ and any $m\ge 1$, we define the following word over $X_{n+k}^{(1)}$:
\[
\mathbf u_{n,k,m}:= x_1x_2\cdots x_{n+k}\,\Bigl(x_nx_{n-1}\cdots x_1\cdot x_{n+1}x_{n+2}\cdots x_{n+k}\Bigr)^{2m-1}.
\]
Notice that $\mathbf u_{n,n,m}=\mathbf v_{n,m}^{(1)}$.

\begin{lemma}
\label{lem:id-BGI}
Let $I$ be a non-empty set, $\mathcal{G}=(G,\cdot)$ a group, and $\mathcal{B}_{G,I}$ the Brandt semigroup over $\mathcal{G}$. For any $m\ge1$ and $n,k\ge 0$ with $n+k>0$ and any substitution $\tau\colon X_{n+k}^{(1)}\to B_{G,I}$, the element $\tau(\mathbf u_{n,k,m})$ lies in a subgroup of $\mathcal{B}_{G,I}$.
\end{lemma}

\begin{proof}
Lemma~\ref{lem:brandt} implies that for any $a\in B_{G,I}$, the power $a^{2m}$ lies in a subgroup of $\mathcal{B}_{G,I}$. Since $\mathbf u_{0,k,m}=(x_1x_2\cdots x_k)^{2m}$ and  $\mathbf u_{1,k,m}=(x_1x_2\cdots x_{k+1})^{2m}$, the claim holds for $n\in\{0,1\}$.

For the rest of the proof, assume $n\ge2$. If $\tau(\mathbf u_{n,k,m})=0$, there is nothing to prove. Let $\tau(\mathbf u_{n,k,m})\ne 0$. Then all elements $\tau(x_i)$, $i=1,2,\dots,n+k$, must lie in the set $I \times G \times I$ of non-zero elements of $\mathcal{B}_{G,I}$. Hence $\tau(x_i)=(\ell_i,g_i,r_i)$ for some $\ell_i,r_i\in I$ and $g_i\in G$.

First, consider the case $k=0$. Since
\[
\mathbf u_{n,0,m}=x_1x_2\cdots x_n\,\Bigl(x_nx_{n-1}\cdots x_1\Bigr)^{2m-1},
\]
we have $\tau(\mathbf u_{n,0,m})=(\ell_1,g,r_1)$ for some $g\in G$. By Lemma~\ref{lem:brandt}, we have to show that $\ell_1=r_1$. We verify that $\ell_j=r_j$ for all $j=n,n-1,\dots,1$ by backwards induction. Since $x_n^2$ occurs as a factor in the word $\mathbf u_{n,0,m}$, we must have $\tau(x_n^2)\ne 0$, whence $\ell_n=r_n$. If $j>1$, both $x_{j-1}x_j$ and $x_jx_{j-1}$ occur as factors in $\mathbf u_{n,0,m}$. We then have
\begin{align*}
r_{j-1}&=\ell_j&&\text{since $\tau(x_{j-1}x_j)\ne 0$,}\\
       &=r_j&&\text{by the induction assumption,}\\
       &=\ell_{j-1}&&\text{since $\tau(x_jx_{j-1})\ne 0$.}
\end{align*}
Now, let $k>0$. Since
\[
\mathbf u_{n,k,m}= x_1x_2\cdots x_{n+k}\,\Bigl(x_nx_{n-1}\cdots x_1\cdot x_{n+1}x_{n+2}\cdots x_{n+k}\Bigr)^{2m-1},
\]
we have $\tau(\mathbf u_{n,k,m})=(\ell_{1},g,r_{n+k})$ for some $g\in G$. Here we have to show that $\ell_1=r_{n+k}$.  Since $\tau(x_1x_2\cdots x_nx_{n+1})\ne0$ and $\tau(x_{n+k}x_n)\ne 0$, we have
\begin{equation}
\label{eq:ascent}
r_{1}=\ell_{2},\,r_{2}=\ell_{3},\dots,r_{n-1}=\ell_{n},\,r_{n}=\ell_{n+1}\ \text{ and  }\ r_{n+k}=\ell_{n}.
\end{equation}
Since $\tau(x_nx_{n-1}\cdots x_1x_{n+1})\ne 0$, we also have
\begin{equation}
\label{eq:descent}
r_{n}=\ell_{n-1},\,r_{n-1}=\ell_{n-2},\dots,r_{2}=\ell_{1},\,r_{1}=\ell_{n+1}.
\end{equation}
Therefore, we obtain
\[
r_{n+k}\stackrel{\eqref{eq:ascent}}{=}\ell_{n}\stackrel{\eqref{eq:ascent}}{=}r_{n-1}\stackrel{\eqref{eq:descent}}{=}\ell_{n-2}\stackrel{\eqref{eq:ascent}}{=}r_{n-3}=\dots=\begin{cases}
\ell_{1}&\text{if $n$ is odd},\\
r_{1}&\text{if $n$ is even}.
\end{cases}
\]
In addition, if $n$ is even, then
\[
r_{1}\stackrel{\eqref{eq:descent}}{=}\ell_{n+1}\stackrel{\eqref{eq:ascent}}{=}r_{n}\stackrel{\eqref{eq:descent}}{=}\ell_{n-1}\stackrel{\eqref{eq:ascent}}{=}r_{n-2}\stackrel{\eqref{eq:descent}}{=}\ell_{n-3}=\dots=\ell_{1}.
\]
We see that $r_{n+k}=\ell_{1}$ in either case.
\end{proof}

Several statements in this section deal with semigroups that possess Brandt series and have subgroups subject to certain restrictions. To keep the premises of these statements compact, we introduce a few short names. In the sequel:
\begin{itemize}
  \item groups of exponent dividing $m$ are called $m$-\emph{groups};
  \item solvable $m$-groups of derived length at most $k$ are called $[m,k]$-\emph{groups};
  \item a semigroup with a Brandt series of height $h$ all of whose subgroups are $m$-groups (resp., $[m,k]$-groups) is called an $(h,m)$-\emph{semigroup} (resp., an $(h,m,k)$-\emph{semigroup}).
\end{itemize}

\begin{lemma}
\label{lem:periodic}
Any $(h,m)$-semigroup satisfies the identity $x^{2^hm}= x^{2^{h+1}m}$.
\end{lemma}

\begin{proof}
Let $(S,\cdot)$ be an $(h,m)$-semigroup and \eqref{eq:series} its Brandt series. We induct on $h$. If $h=0$, then $(S,\cdot)$ is an $m$-group, and any $m$-group satisfies $x^m=x^{2m}$.

Let $h>0$ and let $a$ be an arbitrary element in $S$; we have to verify that
\begin{equation}
\label{eq:period}
a^{2^hm}= a^{2^{h+1}m}.
\end{equation}
If $a$ lies in a subgroup of $(S,\cdot)$, then $a^m=a^{2m}$ since all subgroups of $(S,\cdot)$ are $m$-groups. Squaring this $h$ times gives \eqref{eq:period}. If $a$ lies in no subgroup of $(S,\cdot)$, then $a^2\in S_{h-1}$ --- this is clear if $a\in S_{h-1}$ or the factor $(S_h/S_{h-1},\cdot)$ is a zero semigroup and follows from Lemma~\ref{lem:brandt} if $(S_h/S_{h-1},\cdot)$ is a Brandt semigroup. By the induction assumption, $(a^2)^{2^{h-1}m}= (a^2)^{2^hm}$, that is, \eqref{eq:period} holds again.
\end{proof}

Recall that if $(S,\cdot)$ is a semigroup, $(\langle E(S)\rangle,\cdot)$ stands for its subsemigroup generated by the set $E(S)$ of all idempotents of $(S,\cdot)$.

\begin{lemma}
\label{lem:E(S)}
Let $(S,\cdot)$ be an $(h,m)$-semigroup. For any $k\ge 1$, if a substitution $\tau\colon X_{2n}^{(k)}\to S$ is such that $\tau(x_{i_1i_2\dots i_k})\in\langle E(S)\rangle$ for all $i_1,i_2,\dots,i_k\in\{1,2,\dots,2n\}$, then $\tau(\mathbf v_{n,2^hm}^{(k)})\in E(S)$.
\end{lemma}

\begin{proof}
By Lemma~\ref{lem:periodic} the semigroup $(S,\cdot)$ satisfies $x^{2^hm} = x^{2^{h+1}m}$ whence for each $p\ge 2^hm$ and each $a\in \langle E(S)\rangle$, the elements $a^p$ and $a^{p+1}$ generate the same ideal in the subsemigroup $(\langle E(S)\rangle,\cdot)$. The periodic semigroup $(S,\cdot)$ is a block-group by Lemma~\ref{lem:bg_series} whence the subsemigroup $(\langle E(S)\rangle,\cdot)$ is $\mathscr{J}$-trivial by Proposition \ref{prop:block-group}. Therefore, we have $a^p=a^{p+1}$, i.e., $(\langle E(S)\rangle,\cdot)$ satisfies the identity $x^p=x^{p+1}$.

Letting $a_j:=\begin{cases}\tau(x_j)&\text{if $k=1$},\\
\tau(\mathbf v_{n,2^hm,j}^{(k-1)})&\text{if $k>1$},
\end{cases}$ for $j=1,2,\dots,2n$, we have from \eqref{eq:v1} and \eqref{eq:vnh}
\begin{align*}
\tau(\mathbf v_{n,2^hm}^{(k)})&=a_1a_2\cdots a_{2n}\,\Bigl(a_na_{n-1}\cdots a_1\cdot a_{n+1}a_{n+2}\cdots a_{2n}\Bigr)^{2^{h+1}m-1}&&\\
&=a_1a_2\cdots a_{2n}\,\Bigl(a_1\cdots a_{n-1}a_n\cdot a_{n+1}a_{n+2}\cdots a_{2n}\Bigr)^{2^{h+1}m-1}&&\text{by Lemma~\ref{lem:j-triv}}\\
&=\Bigl(a_1\cdots a_{n-1}a_n\cdot a_{n+1}a_{n+2}\cdots a_{2n}\Bigr)^{2^{h+1}m}&&\\
&=\Bigl(a_1\cdots a_{n-1}a_n\cdot a_{n+1}a_{n+2}\cdots a_{2n}\Bigr)^{2^{h}m}&&\text{as $x^{2^hm} = x^{2^{h+1}m}$}\\
&=\left(\Bigl(a_1\cdots a_{n-1}a_n\cdot a_{n+1}a_{n+2}\cdots a_{2n}\Bigr)^{2^{h}m}\right)^2= \Bigl(\tau(\mathbf v_{n,2^hm}^{(k)})\Bigr)^2.&&
\end{align*}
Hence $\tau(\mathbf v_{n,2^hm}^{(k)})\in E(S)$ as required.
\end{proof}

\begin{lemma}
\label{lem:id-E(S)cupD1}
Let $(S,\cdot)$ be an $(h,m)$-semigroup with Brandt series \eqref{eq:series}. If $(S_1,\cdot)$ is a Brandt semigroup and $S=\langle E(S)\rangle\cup S_1$, then for an arbitrary substitution $\tau\colon X_{2n}^{(1)}\to S$, the element $\tau(\mathbf v_{n,2^hm}^{(1)})$ lies in a subgroup of $S$.
\end{lemma}

\begin{proof}
If $\tau(x_k)\in \langle E(S)\rangle$ for all $k\in\{1,2,\dots,2n\}$, then $\tau(\mathbf v_{n,2^hm}^{(1)})\in E(S)$ by Lemma \ref{lem:E(S)}. Otherwise let $\{k_1,k_2,\dots,k_{p+q}\}$ with
\[
1\le k_1<k_2<\dots<k_p\le n< k_{p+1}<k_{p+2}<\dots<k_{p+q}\le 2n
\]
be the set of all indices $k$ such that $\tau(x_k)\notin \langle E(S)\rangle$. (Here $p=0$ or $q=0$ is possible but $p+q>0$.) Since $S=\langle E(S)\rangle\cup S_1$, we have $\tau(x_{k_1}),\tau(x_{k_2}),\dots,\tau(x_{k_{p+q}})\in S_1$. By Lemma~\ref{lem:fd} and its dual, either $\tau(\mathbf v_{n,2^hm}^{(1)})=0$ or removing all $\tau(x_k)$ such that $\tau(x_k)\in \langle E(S)\rangle$ does not change the value of $\tau(\mathbf v_{n,2^hm}^{(1)})$. If $\tau(\mathbf v_{n,2^hm}^{(1)})=0$, the claim holds. Otherwise, consider the substitution $\tau'\colon X_{p+q}^{(1)}\to S_1$ given by $\tau'(x_s):=\tau(x_{k_s})$ for all $s\in\{1,2,\dots,p+q\}$. Then $\tau(\mathbf v_{n,2^hm}^{(1)})=\tau'(\mathbf u_{p,q,2^hm})$. By Lemma~\ref{lem:id-BGI}, the element $\tau'(\mathbf u_{p,q,2^hm})$ lies in a subgroup of the Brandt semigroup $(S_1,\cdot)$. Therefore, $\tau(\mathbf v_{n,2^hm}^{(1)})$ belongs to a subgroup of $(S,\cdot)$.
\end{proof}

\begin{lemma}
\label{lem:id-G}
Any $[m,k]$-group satisfies the identity $\mathbf v_{n,m}^{(k)}=1$.
\end{lemma}

\begin{proof}
Let $\mathcal{G}=(G,\cdot)$ be an $[m,k]$-group and $a \in G$.  Since $a^m$ is  the identity element of $\mathcal{G}$, we have $a^{2m-1}=a^{-1}$. Consider any substitution $\tau\colon X_{2n}^{(1)}\to G$ and let $a_i:=\tau(x_i)$. Then
\begin{equation}
\label{eq:commutator}
\tau(\mathbf v_{n,m}^{(1)})=a_1a_2\cdots a_{2n}\,\Bigl(a_na_{n-1}\cdots a_1\cdot a_{n+1}a_{n+2}\cdots a_{2n}\Bigr)^{-1}=a_1a_2\cdots a_n\cdot a_1^{-1}a_2^{-1}\cdots a_n^{-1}.
\end{equation}
Recall that the \emph{commutator} of elements $a$ and $b$ of a group is defined as $[a,b]:=a^{-1}b^{-1}ab$. It is easy to verify that the right-hand side of \eqref{eq:commutator} can be written as a product of commutators, namely,
\begin{equation}
\label{eq:prodcommutator}
a_1a_2\cdots a_n\cdot a_1^{-1}a_2^{-1}\cdots a_n^{-1}=[a_1^{-1},a_2^{-1}][(a_2a_1)^{-1},a_3^{-1}]\cdots [(a_{n-1}a_{n-2}\cdots a_1)^{-1},a_n^{-1}].
\end{equation}
Therefore, $\tau(\mathbf v_{n,m}^{(1)})$ belongs to the derived subgroup $(G^{(1)},\cdot)$ of $\mathcal{G}$. We use this as the induction basis and show by induction on $h$ that any substitution $\theta\colon X_{2n}^{(h)}\to G$ sends the word $\mathbf v_{n,m}^{(h)}$ to the $h$-th derived subgroup $(G^{(h)},\cdot)$ of $\mathcal{G}$.

Indeed, let $h>1$. Since for each $j\in\{1,2,\dots,2n\}$, the word $\mathbf v_{n,m,j}^{(h-1)}$ is obtained from the word $\mathbf v_{n,m}^{(h-1)}$ by renaming its variables, $\theta(\mathbf v_{n,m,j}^{(h-1)})\in G^{(h-1)}$ by the induction assumption. Combining \eqref{eq:vnh} and \eqref{eq:commutator}, we see that
\[
\theta(\mathbf v_{n,m}^{(h)})=\theta(\mathbf v_{n,m,1}^{(h-1)})\theta(\mathbf v_{n,m,2}^{(h-1)})\cdots\theta(\mathbf v_{n,m,n}^{(h-1)})\cdot(\theta(\mathbf v_{n,m,1}^{(h-1)}))^{-1}(\theta(\mathbf v_{n,m,2}^{(h-1)}))^{-1}\cdots(\theta(\mathbf v_{n,m,n}^{(h-1)}))^{-1}.
\]
Now rewriting the last expression as in \eqref{eq:prodcommutator}, we see that $\theta(\mathbf v_{n,m}^{(h)})$ is a product of commutators of elements from $G^{(h-1)}$. Hence $\theta(\mathbf v_{n,m}^{(h)})\in G^{(h)}$. Since the group $\mathcal{G}$ is solvable of derived length at most $k$, its $k$-th derived subgroup is trivial, whence the word $\mathbf v_{n,m}^{(k)}$ is sent to  the identity element of $\mathcal{G}$ by any substitution $X_{2n}^{(k)}\to G$. This means that $\mathcal{G}$ satisfies the identity $\mathbf v_{n,m}^{(k)}= 1$.
\end{proof}

\begin{lemma}
\label{lem:id-E(S)cupD2}
Let $(S,\cdot)$ be an $(h,m,k)$-semigroup and \eqref{eq:series} its Brandt series. If $S_0=\{0\}$ and $S=\langle E(S)\rangle\cup S_1$, then $(S,\cdot)$ satisfies the identity $\mathbf v_{n,2^hm}^{(k+1)} = (\mathbf v_{n,2^hm}^{(k+1)})^2$.
\end{lemma}

\begin{proof}
We have to verify that $\tau(\mathbf v_{n,2^km}^{(k+1)})\in E(S)$ for every substitution $\tau\colon X_{2n}^{(k+1)}\to S$. Recall from Remark \ref{rm:vnh} that the word $\mathbf v_{n,2^hm}^{(k+1)}$ is the image of the word $\mathbf v_{n,2^hm}^{(k)}$ under the substitution that sends every variable $x_{i_1\dots i_k}\in X^{(k)}_{2n}$ to the word $\mathbf v_{n,2^hm,i_1,\dots, i_k}^{(1)}$ obtained from $\mathbf v_{n,2^hm}^{(1)}$ by renaming its variables. Hence if all elements of the form $\tau(\mathbf v_{n,2^hm,i_1,\dots, i_k}^{(1)})$ lie in $\langle E(S)\rangle$, then $\tau(\mathbf v_{n,2^hm}^{(k+1)})\in E(S)$ by Lemma \ref{lem:E(S)}. So, we may assume that at least one element of the form $\tau(\mathbf v_{n,2^hm,i_1,\dots, i_k}^{(1)})$ lies in $S_1\setminus\langle E(S)\rangle$.

By the definition of an $(h,m,k)$-semigroup, $(S_1,\cdot)$ is either a zero semigroup or a Brandt semigroup over an $[m,k]$-group. If $(S_1,\cdot)$ is a zero semigroup, we have $\tau(\mathbf v_{n,2^km}^{(k+1)})=0$ since $S_1$ is an ideal of $(S,\cdot)$ and each factor of the form $\mathbf v_{n,2^hm,i_1,\dots, i_k}^{(1)}$ occurs in the word $\mathbf v_{n,2^km}^{(k+1)}$ at least twice. If $(S_1,\cdot)$ is a Brandt semigroup, Lemma~\ref{lem:fd} and its dual imply that either $\tau(\mathbf v_{n,2^hm}^{(k+1)})=0$ or removing all factors $\tau(\mathbf v_{n,2^hm,i_1,\dots, i_k}^{(1)})$ that lie in  $\langle E(S)\rangle$ does not change the value of $\tau(\mathbf v_{n,2^hm}^{(k+1)})$. In the former case, the claim holds, and in the latter case, the remaining elements belong to subgroups of $(S_1,\cdot)$ by Lemma \ref{lem:id-E(S)cupD1}. If some of these elements belong to different maximal subgroups of $(S_1,\cdot)$, then $\tau(\mathbf v_{n,2^hm}^{(k+1)})=0$ by the definition of multiplication in Brandt semigroups. Thus, it remains to consider only the situation when, for all $i_1,i_2,\dots,i_k\in\{1,2,\dots,2n\}$, all elements $\tau(\mathbf v_{n,2^hm,i_1,\dots, i_k}^{(1)})\in S_1\setminus\langle E(S)\rangle$ belong to the same maximal subgroup $(G,\cdot)$ of $(S_1,\cdot)$.

Denote by $e$ the identity element of $(G,\cdot)$ and define the substitution $\overline{\tau}\colon X_{2n}^{(k)}\to G$ by
\[
\overline{\tau}(x_{i_1\dots i_k}):=\begin{cases}
\tau(\mathbf v_{n,2^hm,i_1,\dots, i_k}^{(1)})&\text{if }\tau(\mathbf v_{n,2^hm,i_1,\dots, i_k}^{(1)})\text{ belongs to } G,\\
e&\text{otherwise},
\end{cases}
\]
for each $i_1,i_2,\dots, i_k\in\{1,2,\dots,2n\}$. As we know that $\tau(\mathbf v_{n,2^hm,i_1,\dots, i_k}^{(1)})\in \langle E(S)\rangle\cup G$ for all $i_1,i_2\dots, i_k\in\{1,2,\dots,2n\}$ and $fg=gf=g$ for all $g\in G$, $f\in E(S)$, we conclude that $\tau(\mathbf v_{n,2^hm}^{(k+1)})=\overline{\tau}(\mathbf v_{n,2^hm}^{(k)})$. Since $(G,\cdot)$ is an $[m,k]$-group (and so a $[2^hm,k]$-group), it satisfies the identity $\mathbf v_{n,2^hm}^{(k)} = 1$ by Lemma \ref{lem:id-G}. Hence $\tau(\mathbf v_{n,2^km}^{(k+1)})=\overline{\tau}(\mathbf v_{n,2^hm}^{(k)})=e$. Since the substitution $\tau$ is arbitrary, $(S,\cdot)$ satisfies the identity $\mathbf v_{n,2^hm}^{(k+1)} = (\mathbf v_{n,2^hm}^{(k+1)})^2$.
\end{proof}

\begin{proposition}
\label{prop:id-fh}
Let $(S,\cdot)$ be an $(h,m,k)$-semigroup with Brandt series \eqref{eq:series} and $S_0=\{0\}$. For any $n\ge 2$, $(S,\cdot)$ satisfies the identity
\begin{equation}
\label{eq:IDh}
\mathbf v_{n,2^hm}^{(kh+h)} = (\mathbf v_{n,2^hm}^{(kh+h)})^2.
\end{equation}
\end{proposition}

\begin{proof}
We induct on $h$. If $h=1$, then the claim follows from Lemma~\ref{lem:id-E(S)cupD2}. Let $h>1$. The Rees quotient $(S/S_1,\cdot)$ is an $(h-1,m,k)$-semigroup whose Brandt series starts with the zero term. By the induction assumption, $(S/S_1,\cdot)$ satisfies the identity $\mathbf v_{n,2^{h-1}m}^{(kh-k+h-1)} = (\mathbf v_{n,2^{h-1}m}^{(kh-k+h-1)})^2$ for any $n\ge2$. By Lemma \ref{lem:periodic} $(S/S_1,\cdot)$ satisfies the identity $\mathbf v_{n,2^hm}^{(kh-k+h-1)} = (\mathbf v_{n,2^hm}^{(kh-k+h-1)})^2$ as well. This readily implies that any substitution $X_{2n}^{(kh-k+h-1)}\to S$ sends the word $\mathbf v_{n,2^hm}^{(kh-k+h-1)}$ to either an idempotent in $S\setminus S_1$ or an element in $S_1$. Recall from Remark \ref{rm:vnh} that the word $\mathbf v_{n,2^hm}^{(kh+h)}$ is the image of the word $\mathbf v_{n,2^hm}^{(k+1)}$ under the substitution that sends every variable $x_{i_1\dots i_{k+1}}\in X^{(k+1)}_{2n}$ to the word $\mathbf v_{n,2^hm,i_1,\dots,i_{k+1}}^{(kh-k+h-1)}$ obtained from $\mathbf v_{n,2^hm}^{(kh-k+h-1)}$ by renaming its variables. Therefore, for every substitution $\tau\colon X_{2n}^{(kh+h)}\to S$, all elements of the form $\tau(\mathbf v_{n,2^hm,i_1,\dots,i_{k+1}}^{(kh-k+h-1)})$ lie in either $E(S)$ or $S_1$.

Now consider the substitution $\overline{\tau}\colon X_{2n}^{(k+1)}\to E(S) \cup S_1$ defined by
\[
\overline{\tau}(x_{i_1\dots i_{k+1}}):=\tau(\mathbf v_{n,2^hm,i_1,\dots,i_{k+1}}^{(kh-k+h-1)}),\quad \text{ for each $i_1,i_2,\dots, i_{k+1}\in\{1,2,\dots,2n\}$}.
\]
Clearly, $\tau(\mathbf v_{n,2^hm}^{(kh+h)})=\overline{\tau}(\mathbf v_{n,2^hm}^{(k+1)})$. However, $(\langle E(S) \rangle\cup S_1,\cdot)$ is an $(h,m,k)$-semigroup satisfying the assumption of Lemma \ref{lem:id-E(S)cupD2}. By this lemma, the semigroup $(\langle E(S) \rangle\cup S_1,\cdot)$ satisfies the identity $\mathbf v_{n,2^hm}^{(k+1)}= (\mathbf v_{n,2^hm}^{(k+1)})^2$. This ensures that the element $\tau(\mathbf v_{n,2^hm}^{(kh+h)})=\overline{\tau}(\mathbf v_{n,2^hm}^{(k+1)})$ is an idempotent. Since the substitution $\tau$ is arbitrary, $(S,\cdot)$ satisfies the identity \eqref{eq:IDh}.
\end{proof}

Finally, we remove the restriction $S_0=\{0\}$.

\begin{proposition}
\label{prop:id-fkh+k}
The identity $\mathbf v_{n,2^hm}^{(kh+h+k)} = (\mathbf v_{n,2^hm}^{(kh+h+k)})^2$ with $n\ge 2$ holds in each $(h,m,k)$-semigroup.
\end{proposition}

\begin{proof}
Let $(S,\cdot)$ be an $(h,m,k)$-semigroup with Brandt series \eqref{eq:series}. Consider the Rees quotient $(S/S_0,\cdot)$. This is an $(h,m,k)$-semigroup whose Brandt series starts with the zero term. By Proposition \ref{prop:id-fh}, the semigroup $(S/S_0,\cdot)$ satisfies \eqref{eq:IDh}. Recall from Remark \ref{rm:vnh} that the word $\mathbf v_{n,2^hm}^{(kh+h+k)}$ is the image of the word $\mathbf v_{n,2^hm}^{(k)}$ under the substitution  that sends every variable $x_{i_1\dots i_k}\in X^{(k)}_{2n}$ to the word $\mathbf v_{n,2^hm,i_1,i_2,\dots,i_k}^{(hk+h)}$ obtained from $\mathbf v_{n,2^hm}^{(hk+h)}$ by renaming its variables. Then, for every substitution $\tau\colon X_{2n}^{(kh+h+k)}\to S$, the element $\tau(\mathbf v_{n,2^hm,i_1,\dots,i_k}^{(kh+h)})$ is an idempotent of $(S/S_0,\cdot)$. Therefore, $\tau(\mathbf v_{n,2^hm,i_1,\dots,i_k}^{(kh+h)})$ lies in either $E(S)\setminus S_0$ or $S_0$. Recall that by the definition of an $(h,m,k)$-semigroup, $(S_0,\cdot)$ is an $[m,k]$-group. Denote by $e$ the identity element of this group. For each $f\in E(S)$, the product $fe$ lies in $S_0$ because $S_0$ is an ideal in $S$. Then $fe$ is an idempotent in $S_0$ by Lemma \ref{lem:ef-idempotent} which applies since  $(S,\cdot)$ is a periodic block-group by Lemmas~\ref{lem:periodic} and~\ref{lem:bg_series}. Hence $fe=e$ since a group has no idempotents except its identity element. Consequently, for every $g\in S_0$, we have $fg=feg=eg=g$. Dually, $gf=g$ for all $g\in S_0$, $f\in E(S)$. Now consider the substitution $\overline{\tau}\colon X_{2n}^{(k)}\to S_0$ defined by
\[
\overline{\tau}(x_{i_1\dots i_k}):=\begin{cases}
\tau(\mathbf v_{n,2^hm,i_1,\dots, i_k}^{(kh+h)})&\text{if }\tau(\mathbf v_{n,m,i_1,\dots, i_k}^{(kh+h)})\text{ belongs to } S_0,\\
e&\text{otherwise},
\end{cases}
\]
for each $i_1,i_2,\dots, i_k\in\{1,2,\dots,2n\}$.
As we know that $\tau(\mathbf v_{n,2^hm,i_1,\dots, i_k}^{(kh+h)})\in E(S)\cup S_0$ for all $i_1,i_2\dots, i_k\in\{1,2,\dots,2n\}$ and $fg=gf=g$ for all $g\in S_0$, $f\in E(S)$, we conclude that $\tau(\mathbf v_{n,2^hm}^{(kh+h+k)})=\overline{\tau}(\mathbf v_{n,2^hm}^{(k)})$. Since $(S_0,\cdot)$ is an $[m,k]$-group (and so a $[2^hm,k]$-group), it satisfies the identity $\mathbf v_{n,2^hm}^{(k)} = 1$ by Lemma \ref{lem:id-G}. Hence $\tau(\mathbf v_{n,2^hm}^{(kh+h+k)})=\overline{\tau}(\mathbf v_{n,2^hm}^{(k)})=e$. Since the substitution $\tau$ is arbitrary, $(S,\cdot)$ satisfies the identity $\mathbf v_{n,2^hm}^{(kh+h+k)} = (\mathbf v_{n,2^hm}^{(kh+h+k)})^2$.
\end{proof}

Now we are ready to complete the mission of this section.

\begin{proposition}
\label{prop:finbg}
For each finite block-group $(S,\cdot)$ with solvable subgroups, there exist positive integers $q$ and $r$ such that $(S,\cdot)$ satisfies the identities $\mathbf v_{n,q}^{(r)} = (\mathbf v_{n,q}^{(r)})^2$ for all $n\ge 2$.
\end{proposition}

\begin{proof}
Clearly, each finite semigroup possesses a principal series. By Corollary~\ref{cor:series}, every principal series of a block-group is a Brandt series. Let $h$ be the height of a Brandt series of $(S,\cdot)$, let $m$ be the least common multiple of the exponents of subgroups of $(S,\cdot)$, and let $k$ be the maximum of the derived lengths of these subgroups. Then $(S,\cdot)$ is an $(h,m,k)$-semi\-group. Now Proposition~\ref{prop:id-fkh+k} applies, showing that $(S,\cdot)$ satisfies all identities $\mathbf v_{n,2^hm}^{(kh+h+k)} = (\mathbf v_{n,2^hm}^{(kh+h+k)})^2$ with $n\ge 2$. Hence one can choose $q=2^hm$ and $r=kh+k+h$.
\end{proof}

Proposition~\ref{prop:finbg} gives a vast generalization of \cite[Proposition 2.4]{GV22} where a similar result was established for finite inverse semigroups with abelian subgroups.

\section{Proofs of Theorems~\ref{thm:psg} and~\ref{thm:hall}}
\label{sec:proofs}

Recall that the 6-\emph{element Brandt monoid} $(B_2^1,\cdot)$ is formed by the following zero-one $2\times 2$-matrices
\begin{equation}\label{eq:brandt}
\begin{tabular}{cccccc}
$\left(\begin{matrix} 0&0\\0&0\end{matrix}\right)$
&
$\left(\begin{matrix} 1&0\\0&1\end{matrix}\right)$
&
$\left(\begin{matrix} 0&1\\0&0\end{matrix}\right)$
&
$\left(\begin{matrix} 0&0\\1&0\end{matrix}\right)$
&
$\left(\begin{matrix} 1&0\\0&0\end{matrix}\right)$
&
$\left(\begin{matrix} 0&0\\0&1\end{matrix}\right)$
\end{tabular}
\end{equation}
under the usual matrix multiplication $\cdot$. In terms of the construction for Brandt semigroups presented in Sect.~\ref{sec:prelim}, $(B_2^1,\cdot)$  is obtained by adjoining the identity element to the Brandt semigroup over the trivial group with the 2-element index set. It is known and easy to verify that the set $B_2^1$ admits a unique addition $+$ such that $\mathcal{B}_2^1:=(B_2^1,+,\cdot)$ becomes an \ais. The addition is nothing but the Hadamard (entry-wise) product of matrices:
\[
(a_{ij})+(b_{ij}):=(a_{ij}b_{ij}).
\]
The identities of the semiring $\mathcal{B}_2^1$ were studied by the second author \cite{Vol21} and, independently, by Jackson, Ren, and Zhao~\cite{JackRenZhao}; it was proved that these identities admit no finite basis. Moreover, in~\cite{GV22} we proved that this property of $\mathcal{B}_2^1$ is very contagious: if all identities of an \ais\ hold in $\mathcal{B}_2^1$, then under some rather mild conditions, the identities have no finite basis. Here is the key result of~\cite{GV22} stated in the notation adopted here.

\begin{theorem}[{\!\cite[Theorem 4.2]{GV22}}]
\label{thm:old}
Let $\mathcal{S}$ be an \ais\ whose multiplicative reduct satisfies the identities $\mathbf v_{n,q}^{(r)} = (\mathbf v_{n,q}^{(r)})^2$ for all $n\ge 2$ and some $q,r\ge 1$. If all identities of $\mathcal{S}$ hold in the \ais\ $\mathcal{B}_2^1$, then the identities admit no finite basis.
\end{theorem}

In~\cite{GV22}, Theorem~\ref{thm:old} was applied to finite \ais{}s whose multiplicative reducts were inverse semigroups with abelian subgroups. Due to Proposition~\ref{prop:finbg}, we are in a position to enlarge its application range as follows.

\begin{theorem}
\label{thm:main}
Let $\mathcal{S}=(S,+,\cdot)$ be a finite \ais\ whose multiplicative reduct is a block-group with solvable subgroups. If all identities of $\mathcal{S}$ hold in the \ais\ $\mathcal{B}_2^1$, then the identities admit no finite basis.
\end{theorem}

\begin{proof}
Proposition \ref{prop:finbg} implies that the semigroup $(S,\cdot)$ satisfies the identities $\mathbf v_{n,q}^{(r)} = (\mathbf v_{n,q}^{(r)})^2$ for all $n\ge 2$ and some $q,r\ge 1$, and therefore, Theorem~\ref{thm:old} applies.
\end{proof}

We are ready to deduce Theorems~\ref{thm:psg} and~\ref{thm:hall}.

\begin{proof}[Proof of Theorem~\ref{thm:psg}]
We aim to prove that for any finite non-abelian solvable group $\mathcal{G}=(G,\cdot)$, the identities of the power group $(\mathcal{P}(G),\cup,\cdot)$ admit no finite basis. It is shown in \cite[Proposition 2.4]{Pin:1980}, see also \cite[Proposition 11.1.2]{Almeida:95}, that the semigroup $(\mathcal{P}(G),\cdot)$ is a block-group and all maximal subgroups of $(\mathcal{P}(G),\cdot)$ are of the form $(N_G(H)/H,\cdot)$ where $(H,\cdot)$ is a subgroup of $\mathcal{G}$ and $N_G(H)=\{g\in G\mid gH=Hg\}$ is the normalizer of $(H,\cdot)$ in $\mathcal{G}$. Hence each subgroup of $(\mathcal{P}(G),\cdot)$ embeds into a quotient of a subgroup of $\mathcal{G}$ and so it is solvable. We see that the multiplicative reduct of $(\mathcal{P}(G),\cup,\cdot)$ is a block-group with solvable subgroups.

First suppose that $\mathcal{G}$ is not a Dedekind group, that is, some subgroup $(H,\cdot)$ is not normal in $\mathcal{G}$. Then, of course, $1<|H|<|G|$ and one can find an element $g\in G$ with $g^{-1}Hg\ne H$. Denote by $E$ the singleton subgroup of $\mathcal{G}$, let $J:=\{A\subseteq G\mid |A|>|H|\}$, and consider the following subset of $\mathcal{P}(G)$:
\begin{equation}\label{eq:subsetb}
B:=\{E,\,H,\,g^{-1}H,\,Hg,\,g^{-1}Hg\}\cup J.
\end{equation}
Observe that $g^{-1}H\ne Hg$: indeed $g^{-1}H\cdot Hg=g^{-1}Hg$ while $Hg\cdot g^{-1}H=H$, so that the assumption $g^{-1}H=Hg$ forces the equality $g^{-1}Hg=H$ that contradicts the choice of $g$. It is easy to verify that $B$ is closed under union and element-wise multiplication of subsets so that $(B,\cup,\cdot)$ is a subsemiring in $(\mathcal{P}(G),\cup,\cdot)$. Define a map $B\to B_2^1$ via the following rule:
\begin{equation}
\label{eq:map}
\begin{tabular}{cccccc}
$A\in J$ & $E$ & $Hg$ & $g^{-1}H$ & $H$ & $g^{-1}Hg$\\
$\downarrow$ & $\downarrow$ & $\downarrow$ & $\downarrow$ & $\downarrow$ & $\downarrow$ \\
$\left(\begin{matrix} 0&0\\0&0\end{matrix}\right)$
&
$\left(\begin{matrix} 1&0\\0&1\end{matrix}\right)$
&
$\left(\begin{matrix} 0&1\\0&0\end{matrix}\right)$
&
$\left(\begin{matrix} 0&0\\1&0\end{matrix}\right)$
&
$\left(\begin{matrix} 1&0\\0&0\end{matrix}\right)$
&
$\left(\begin{matrix} 0&0\\0&1\end{matrix}\right)$
\end{tabular}.
\end{equation}
In can be routinely verified that this map is a homomorphism from the subsemiring $(B,\cup,\cdot)$  onto the \ais\ $\mathcal{B}_2^1$. Since $\mathcal{B}_2^1$ is homomorphic image of a subsemiring in $(\mathcal{P}(G),\cup,\cdot)$, all identities of $(\mathcal{P}(G),\cup,\cdot)$ hold in $\mathcal{B}_2^1$. Thus, Theorem~\ref{thm:main} applies to $(\mathcal{P}(G),\cup,\cdot)$, showing that its identities admit no finite basis.

\medskip

If $\mathcal{G}$ is a Dedekind group, then since $\mathcal{G}$ is non-abelian, it has a subgroup isomorphic to the 8-element quaternion group; see \cite[Theorem 12.5.4]{Hall:1959}. This subgroup is isomorphic to a subgroup in $(\mathcal{P}(G),\cdot)$, and by~\cite[Theorem~6.1]{JackRenZhao} the identities of every \ais\ whose multiplicative reduct possesses a non-abelian nilpotent subgroup admit no finite basis. We see that $(\mathcal{P}(G),\cup,\cdot)$ has no finite identity basis in this case as well.
\end{proof}

\begin{remark}
Given a semigroup $(S,\cdot)$, one sometimes considers, along with the power semigroup $(\mathcal{P}(S),\cdot)$, its subsemigroup $(\mathcal{P}'(S),\cdot)$ where $\mathcal{P}'(S)$ is the set of all \emph{non-empty} subsets of $S$; see, e.g., \cite[Chapter 11]{Almeida:95}. Clearly, the empty set serves as a zero in power semigroups so that one can think of $(\mathcal{P}(S),\cdot)$ as the result of adjoining a new zero to $(\mathcal{P}'(S),\cdot)$.
It is well known that the property of a semigroup to have or not to have a finite identity basis does not change when a new zero is adjoined; see \cite[Section 3]{Vo01}. Therefore $(\mathcal{P}(S),\cdot)$ and $(\mathcal{P}'(S),\cdot)$ admit or do not admit a finite identity basis simultaneously.

Since the set $\mathcal{P}'(S)$ is closed under union, one can consider the semiring $(\mathcal{P}'(S),\cup,\cdot)$. The power semiring $(\mathcal{P}(S),\cup,\cdot)$ can be thought of as the result of adjoining a new element, which is both an additive identity and multiplicative zero, to $(\mathcal{P}'(S),\cup,\cdot)$. However, it is not yet known whether adjoining such an element preserves the property of having/not having a finite identity basis.
Nevertheless, the above proof of Theorem~\ref{thm:psg} applies without a hitch to \ais{}s of the form $(\mathcal{P}'(G),\cup,\cdot)$ where $(G,\cdot)$ is any finite non-abelian solvable group and shows that the identities of every such \ais\ admit no finite basis.
\end{remark}

\begin{remark}
It is quite natural to consider $(\mathcal{P}(S),\cup,\cdot,\varnothing)$ as an algebra of type $(2,2,0)$; this was actually the setting adopted by Dolinka~\cite{Dolinka09a}. One can verify that the conclusion of  Theorem~\ref{thm:psg} persists in this case as well. The `non-Dedekind' part of the proof of Theorem~\ref{thm:psg} invokes our results from \cite{GV22} via Theorem~\ref{thm:main}. These results work fine for \ais{}s with multiplicative zero treated as algebras of type $(2,2,0)$; this is explicitly registered in \cite[Remark 3]{GV22}. In the part of the proof that deals with Dedekind groups, one should refer to Theorem~7.6 from \cite{JackRenZhao} that transfers the result of~\cite[Theorem~6.1]{JackRenZhao} to the signature $\{+,\cdot,0\}$.
\end{remark}

\begin{proof}[Proof of Theorem~\ref{thm:hall}]
The theorem claims that the identities of the semiring $(H(X),\cup,\cdot)$ of all Hall relations on a finite set $X$ admit a finite basis if and only if $|X|=1$.  The `if' part is obvious, since if $|X|=1$, then $|H(X)|=1$, and the identities of the trivial semiring have a basis consisting of the single identity $x=y$.

For the `only if' part, assume that $X=X_n:=\{1,2,\dots,n\}$ and $n\ge2$. The semigroup $(H(X_n),\cdot)$ is a block-group; see, e.g., \cite[Proposition 2]{GV21}. It can be easily verified that the \ais\ $\mathcal{B}_2^1$ embeds into $(H(X_2),\cup,\cdot)$ via the following map:
\[
\begin{tabular}{cccccc}
$\left(\begin{matrix} 0&0\\0&0\end{matrix}\right)$
&
$\left(\begin{matrix} 1&0\\0&1\end{matrix}\right)$
&
$\left(\begin{matrix} 0&1\\0&0\end{matrix}\right)$
&
$\left(\begin{matrix} 0&0\\1&0\end{matrix}\right)$
&
$\left(\begin{matrix} 1&0\\0&0\end{matrix}\right)$
&
$\left(\begin{matrix} 0&0\\0&1\end{matrix}\right)$\\
$\downarrow$ & $\downarrow$ & $\downarrow$ & $\downarrow$ & $\downarrow$ & $\downarrow$ \\
$\left(\begin{matrix} 1&1\\1&1\end{matrix}\right)$
&
$\left(\begin{matrix} 1&0\\0&1\end{matrix}\right)$
&
$\left(\begin{matrix} 0&1\\1&1\end{matrix}\right)$
&
$\left(\begin{matrix} 1&1\\1&0\end{matrix}\right)$
&
$\left(\begin{matrix} 1&0\\1&1\end{matrix}\right)$
&
$\left(\begin{matrix} 1&1\\0&1\end{matrix}\right)$
\end{tabular}.
\]
The matrices in the bottom row are matrices over the Boolean semiring $\mathbb{B}:=(\{0,1\},+,\cdot)$ with $0\cdot 0=0\cdot 1=1\cdot 0=0+0=0$, $1\cdot 1=1+0=0+1=1+1=1$, and we mean the standard representation of binary relations as matrices over the Boolean semiring (in which Hall relations correspond to matrices with permanent~1). Obviously, the \ais\ $(H(X_2),\cup,\cdot)$ embeds into $(H(X_n),\cup,\cdot)$ for each $n\ge 2$, whence so does $\mathcal{B}_2^1$. Therefore all identities of $(H(X_n),\cup,\cdot)$ hold in $\mathcal{B}_2^1$.

Let $\mathrm{Sym}_n$ stand for the group of all permutations of the set $X_n$. It is known that all subgroups of the semigroup of all binary relations on $X_n$ embed into $\mathrm{Sym}_n$; see, e.g., \cite[Theorem 3.5]{MP69}.
In particular, all subgroups of the semigroup $(H(X_n),\cdot)$ with $n\le 4$ embed into $\mathrm{Sym}_4$, and hence, are solvable. We see that  Theorem~\ref{thm:main} applies to the \ais\ $(H(X_n),\cup,\cdot)$ with $n=2,3,4$, whence the identities of the \ais\ admit no finite basis.

To cover the case $n>4$, observe that the semigroup $(H(X_n),\cdot)$ has $\mathrm{Sym}_n$ as its group of units, and for $n>4$ (even for $n\ge4$), this group possesses non-abelian nilpotent subgroups, for instance, the 8-element dihedral group. By~\cite[Theorem~6.1]{JackRenZhao} the \ais\ $(H(X_n),\cup,\cdot)$ has no finite identity basis in this case as well.
\end{proof}

In the introduction, we mentioned Dolinka's problem \cite[Problem 6.5]{Dolinka09a} as the main motivation of our study. Recall that the problem asks for a description of finite groups whose power groups have no finite identity basis. Combining Theorem~\ref{thm:psg} of the present paper and Theorem~6.1 of \cite{JackRenZhao} reduces the problem to two special cases.

First, the problem remains open for non-solvable groups whose nilpotent subgroups are abelian. The structure of such groups is well understood, see~\cite{Br71}; an interesting example here is the sporadic simple group of order 175560 found by Janko~\cite{Ja66}. While arguments in the proof of Theorem~\ref{thm:psg} do not cover this case, we conjecture that the conclusion of the theorem persists. In other words, we believe that the power group of any finite non-abelian group has no finite identity basis.

Second, the problem is open for a majority of abelian groups. This may seem surprising as power groups of abelian groups have commutative multiplication, but finite \ais{}s\ with commutative multiplication and without finite identity basis are known to exist, see \cite[Corollary 4.11]{JackRenZhao}. This does not exclude, of course, that all power groups of finite abelian groups admit a finite identity basis, but to the best of our knowledge, the only group for which such a basis is known is the 2-element group. Namely, Ren and Zhao~\cite{RZ} have shown that the identities $xy=yx$ and $x=x^3$ form a basis for identities of a 4-element \ais\ that can be easily identified with the power group of the 2-element group.

\section{Proof of Theorem~\ref{thm:inv-psg}}
\label{sec:involution}

Recall that in an inverse semigroup $(S,\cdot)$, every element has a unique inverse; the inverse of $a\in S$ is denoted $a^{-1}$. This defines a unary operation on $S$, and it is well known (and easy to verify) that
$(S,\cdot,{}^{-1})$ is an involution semigroup. In particular, the 6-element Brandt monoid can be considered as an involution semigroup. (Observe that in terms of the matrix representation \eqref{eq:brandt} the involution of $(B_2^1,\cdot,{}^{-1})$ is just the usual matrix transposition.)

We need the following result that translates Theorem~\ref{thm:old} into the language of involution semigroups. Its proof is parallel to that of \cite[Theorem 4.2]{GV22}.

\begin{theorem}
\label{thm:inv}
Let $\mathcal{S}=(S,\cdot,{}^\ast)$ be an involution semigroup whose multiplicative reduct satisfies the identities $\mathbf v_{n,q}^{(r)} = (\mathbf v_{n,q}^{(r)})^2$ for all $n\ge 2$ and some $q,r\ge 1$. If all identities of $\mathcal{S}$ hold in the involution semigroup $(B_2^1,\cdot,{}^{-1})$, then the identities admit no finite basis.
\end{theorem}

\begin{proof}
As in the proof of  \cite[Theorem 4.2]{GV22}, we employ a family of inverse semigroups constructed by Ka\softd{}ourek in~\cite[Section~2]{Kad03}. First, for all $n,h\ge 1$, define terms $\mathbf w_n^{(h)}$ of the signature $(\cdot,{}^{-1})$ over the alphabet $X_n^{(h)}=\{x_{i_1i_2\cdots i_h}\mid i_1,i_2,\dots,i_h\in\{1,2,\dots,n\}\}$ by induction on $h$. Put
\[
\mathbf w_n^{(1)}:=x_1x_2\cdots x_n\,x_1^{-1}x_2^{-1}\cdots x_n^{-1}.
\]
Then, assuming that, for any $h>1$, the unary term $\mathbf w_n^{(h-1)}$ over the alphabet $X_n^{(h-1)}$ has already been defined, we create $n$ copies of this term over the alphabet $X_n^{(h)}$ as follows. For every $j\in\{1,2,\dots,n\}$, we put
\[
\mathbf w_{n,j}^{(h-1)}:=\sigma_{n,j}^{(h)}(\mathbf w_n^{(h-1)}),
\]
where the substitution $\sigma_{n,j}^{(h)}\colon X_{n}^{(h-1)}\to X_{n}^{(h)}$ appends $j$ to the indices of its arguments, i.e.,
\[
\sigma_{n,j}^{(h)}(x_{i_1i_2\dots i_{h-1}}):= x_{i_1i_2\dots i_{h-1}j} \text{ for all } i_1,i_2,\dots,i_{h-1}\in\{1,2,\dots,n\}.
\]
Then we put
\[
\mathbf w_n^{(h)}:=\mathbf w_{n,1}^{(h-1)}\mathbf w_{n,2}^{(h-1)}\cdots \mathbf w_{n,n}^{(h-1)}(\mathbf w_{n,1}^{(h-1)})^{-1}(\mathbf w_{n,2}^{(h-1)})^{-1}\cdots (\mathbf w_{n,n}^{(h-1)})^{-1}.
\]
Now for any $n\ge2$ and $h\ge1$, let $(S_n^{(h)},\cdot,{}^{-1})$ stand for the inverse semigroup of partial one-to-one transformations on the set $\{0,1,\dots,2^hn^h\}$ generated by $n^h$ transformations $\chi_{i_1i_2\dots i_h}$ with arbitrary indices $i_1,i_2,\dots, i_h\in\{1,2,\dots,n\}$ defined as follows:
\begin{itemize}
\item $\chi_{i_1i_2\dots i_h}(p-1)=p$ if and only if in the term $\mathbf w_n^{(h)}$, the element in the $p$-th position from the left is $x_{i_1i_2\dots i_h}$;
\item $\chi_{i_1i_2\dots i_h}(p)=p-1$ if and only if in the term $\mathbf w_n^{(h)}$, the element in the $p$-th position from the left is $x_{i_1i_2\dots i_h}^{-1}$.
\end{itemize}

We borrow the following illustrative example from~\cite{GV22}. Let $n=2$, $h=2$; then
\[
\mathbf w_2^{(2)}=\underbrace{x_{11}x_{21}x_{11}^{-1}x_{21}^{-1}}_{\mathbf w_{2,1}^{(1)}}\ \underbrace{x_{12}x_{22}x_{12}^{-1}x_{22}^{-1}}_{\mathbf w_{2,2}^{(1)}}\ \underbrace{x_{21}x_{11}x_{21}^{-1}x_{11}^{-1}}_{(\mathbf w_{2,1}^{(1)})^{-1}}\ \underbrace{x_{22}x_{12}x_{22}^{-1}x_{12}^{-1}}_{(\mathbf w_{2,2}^{(1)})^{-1}},
\]
and the generators of the inverse semigroup $(S_2^{(2)},\cdot,{}^{-1})$ act as shown in Figure~\ref{fig:kadourek}.
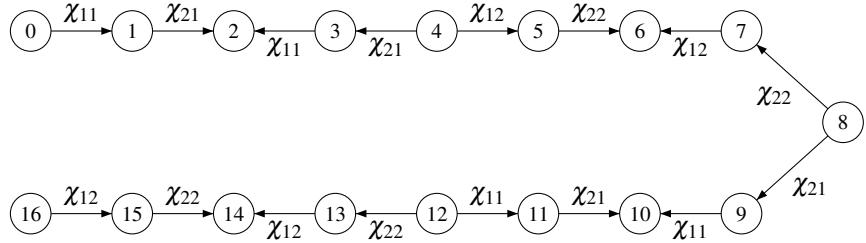
\begin{figure}[h]
\begin{center}
\unitlength 1.35mm
\begin{picture}(80,25)(0,-2.5)
\gasset{Nw=4,Nh=4,Nmr=2}
\node(u0)(0,18){\footnotesize 0}
\node(u1)(10,18){\footnotesize 1}
\node(u2)(20,18){\footnotesize 2}
\node(u3)(30,18){\footnotesize 3}
\node(u4)(40,18){\footnotesize 4}
\node(u5)(50,18){\footnotesize 5}
\node(u6)(60,18){\footnotesize 6}
\node(u7)(70,18){\footnotesize 7}
\node(u8)(80,9){\footnotesize 8}
\node(u9)(70,0){\footnotesize 9}
\node(u10)(60,0){\footnotesize 10}
\node(u11)(50,0){\footnotesize 11}
\node(u12)(40,0){\footnotesize 12}
\node(u13)(30,0){\footnotesize 13}
\node(u14)(20,0){\footnotesize 14}
\node(u15)(10,0){\footnotesize 15}
\node(u16)(0,0){\footnotesize 16}
\drawedge(u0,u1){$\chi_{11}$}
\drawedge(u1,u2){$\chi_{21}$}
\drawedge(u3,u2){$\chi_{11}$}
\drawedge(u4,u3){$\chi_{21}$}
\drawedge(u4,u5){$\chi_{12}$}
\drawedge(u5,u6){$\chi_{22}$}
\drawedge(u7,u6){$\chi_{12}$}
\drawedge(u8,u7){$\chi_{22}$}
\drawedge(u8,u9){$\chi_{21}$}
\drawedge(u9,u10){$\chi_{11}$}
\drawedge(u11,u10){$\chi_{21}$}
\drawedge(u12,u11){$\chi_{11}$}
\drawedge(u12,u13){$\chi_{22}$}
\drawedge(u13,u14){$\chi_{12}$}
\drawedge(u15,u14){$\chi_{22}$}
\drawedge(u16,u15){$\chi_{12}$}
\end{picture}
\caption{The generators of the inverse semigroup $(S_2^{(2)},\cdot,{}^{-1})$}\label{fig:kadourek}
\end{center}
\end{figure}

We need two properties of the inverse semigroups $(S_n^{(h)},\cdot,{}^{-1})$.
\begin{itemize}
  \item[(A)] \cite[Corollary~3.2]{Kad03} For each $n\ge 2$ and $h\ge 1$, any inverse subsemigroup of $(S_n^{(h)},\cdot,{}^{-1})$ generated by less than $n$ elements satisfies all identities of the involution semigroup $(B_2^1,\cdot,{}^{-1})$.
  \item[(B)] \cite[Proposition 3.3]{GV22} For each $n\ge 2$ and $h\ge 1$ and any $m\ge 1$, the semigroup $(S_n^{(h)},\cdot)$ violates the identity $\mathbf v_{n,m}^{(h)}=(\mathbf v_{n,m}^{(h)})^2$.
\end{itemize}

Now, arguing by contradiction, assume that for some $k$ the involution semigroup $\mathcal{S}$ has an identity basis $\Sigma$ such that each identity in $\Sigma$ involves less than $k$ variables. Fix an arbitrary identity $\mathbf u=\mathbf v$ in $\Sigma$ and let $x_1,x_2,\dots,x_\ell$ be all variables that occur in $\mathbf u$ or $\mathbf v$. Consider an arbitrary substitution $\tau\colon\{x_1,x_2,\dots,x_\ell\}\to S_k^{(r)}$ and let $(T,\cdot,{}^{-1})$ be the inverse subsemigroup of $(S_k^{(r)},\cdot,{}^{-1})$ generated by the elements $\tau(x_1),\tau(x_2),\dots,\tau(x_\ell)$. Since $\ell<k$, Property (A) implies that $(T,\cdot,{}^{-1})$ satisfies all identities of the involution semigroup $(B_2^1,\cdot,{}^{-1})$. Since by the condition of the theorem, $(B_2^1,\cdot,{}^{-1})$ satisfies all identities of $\mathcal{S}$, the identity $\mathbf u=\mathbf v$ holds in $(B_2^1,\cdot,{}^{-1})$. Hence the identity $\mathbf u=\mathbf v$ holds also in the involution semigroup $(T,\cdot,{}^{-1})$, and so $\mathbf u$ and $\mathbf v$ take the same value under every substitution of elements of $T$ for the variables $x_1,\dots,x_\ell$. In particular, $\tau(\mathbf u)=\tau(\mathbf v)$. Since the substitution $\tau$ is arbitrary, the identity $\mathbf u=\mathbf v$ holds in the involution semigroup $(S_k^{(r)},\cdot,{}^{-1})$. Since $\mathbf u=\mathbf v$ is an arbitrary identity from the identity basis $\Sigma$ of $\mathcal{S}$, we see that $(S_k^{(r)},\cdot,{}^{-1})$ satisfies all identities of $\mathcal{S}$. In particular,  the multiplicative reduct $(S_k^{(r)},\cdot)$ satisfies all identities of the multiplicative reduct $(S,\cdot)$ of $\mathcal{S}$. By the condition of the theorem, $(S,\cdot)$ satisfies the identity $\mathbf v_{k,q}^{(r)} = (\mathbf v_{k,q}^{(r)})^2$. On the other hand, Property (B) shows that this identity fails in $(S_k^{(r)},\cdot)$, a contradiction.
\end{proof}

Now, we deduce Theorem~\ref{thm:inv-psg} from Theorem~\ref{thm:inv}, applying exactly the same construction that was used in the `non-Dedekind' part of the proof of Theorem~\ref{thm:psg}.

\begin{proof}[Proof of Theorem~\ref{thm:inv-psg}] We aim to prove that for any finite non-Dedekind solvable group $\mathcal{G}=(G,\cdot)$, the identities of the involution semigroup $(\mathcal{P}(G),\cdot,{}^{-1})$ admit no finite basis. As observed in the first paragraph of the proof of of Theorem~\ref{thm:psg}, the multiplicative reduct $(\mathcal{P}(G),\cdot)$ is a block-group with solvable subgroups. By Proposition~\ref{prop:finbg}, the reduct satisfies the identities $\mathbf v_{n,q}^{(r)} = (\mathbf v_{n,q}^{(r)})^2$ for all $n\ge 2$ and some $q,r\ge 1$.

As the group $(G,\cdot)$ is not Dedekind, one can choose a subgroup $(H,\cdot)$ of $\mathcal{G}$ and an element $g\in G$ such that $g^{-1}Hg\ne H$. Consider the subset $B$ of $\mathcal{P}(G)$ defined by  \eqref{eq:subsetb}.
Applying the element-wise inversion to the elements of $B$, we see that
\[
E^{-1}=E,\ H^{-1}=H,\ (g^{-1}H)^{-1}=Hg,\ (Hg)^{-1}=g^{-1}H,\ (g^{-1}Hg)^{-1}=g^{-1}Hg,
\]
and $A^{-1}\in J$ for all $A\in J$. Hence $(B,\cdot,{}^{-1})$ is an involution subsemigroup of $(\mathcal{P}(G),\cdot,{}^{-1})$ and the map \eqref{eq:map} is a homomorphism of this involution subsemigroup onto  $(B_2^1,\cdot,{}^{-1})$. Hence, all identities of $(\mathcal{P}(G),\cdot,{}^{-1})$ hold in the involution semigroup $(B_2^1,\cdot,{}^{-1})$. Now Theorem~\ref{thm:inv} applies to $(\mathcal{P}(G),\cdot,{}^{-1})$ and shows that its identities have no finite basis.
\end{proof}

\begin{remark}
The analogy between Theorems~\ref{thm:psg} and~\ref{thm:inv-psg} is not complete since the latter theorem does not cover finite groups that are Dedekind but non-abelian. In particular, we do not know whether the involution power semigroup of the 8-element quaternion group admits a finite identity basis.
\end{remark}

\begin{remark}
The semigroup of all binary relations on a set $X$ has a useful involution: for each $\rho\subseteq X\times X$, one sets $\rho^{-1}:=\{(y,x)\mid (x,y)\in\rho\}$. The set $H(X)$ of all Hall relations on $X$ is closed under this involution so that one can consider the involution semigroup $(H(X),\cdot,{}^{-1})$. The reader might wonder if a result related to Theorem~\ref{thm:hall} as Theorem~\ref{thm:inv-psg} relates to Theorem~\ref{thm:psg} holds, and if it does, why it is missing in this paper. In fact, the exact analog of Theorem~\ref{thm:hall} does hold: the identities of the involution semigroup of all Hall relations on a finite set $X$ admit a finite basis if and only if $|X|=1$. This follows from~\cite[Remark 3.6]{ADV12} that establishes a stronger fact that we briefly describe now.

Recall that an algebra of some type $\tau$ is said to be \emph{weakly finitely based} if it fulfills a finite set $\Sigma$ of identities such that every finitely generated algebra of type $\tau$ that satisfies $\Sigma$ is finite. An algebra that is not weakly finitely based is called \emph{inherently nonfinitely based}. In~\cite[Remark 3.6]{ADV12} it is shown that the involution semigroup $(H(X),\cdot,{}^{-1})$ with $|X|>1$ is inherently nonfinitely based. This property is known to be much stronger than the absence of a finite identity basis and it cannot be achieved via the approach of the present paper.

In contrast, Dolinka~\cite[Theorem~6]{Dolinka10a} has proved that for every finite group $(G,\cdot)$, the involution semigroup $(\mathcal{P}(G),\cdot,{}^{-1})$ is weakly finitely based. It means that the Finite Basis Problem for involution power semigroups of groups is out of the range of methods developed in \cite{Dolinka10,ADV12}, and therefore, this justifies including Theorem~\ref{thm:inv-psg}.
\end{remark}

\begin{remark}
An \emph{involution \ais} is an algebra $(S, +, \cdot,{}^*)$ of type $(2,2,1)$  such that $(S, +, \cdot)$ is an \ais, $(S,\cdot,{}^*)$ is an involution semigroup and, besides that, the law
\[
(x+y)^*=x^*+y^*
\]
holds. It is easy to see that $(\mathcal{P}(G),\cup,\cdot,{}^{-1})$ is an involution \ais\ for any group $\mathcal{G}=(G,\cdot)$. The above proofs of Theorems~\ref{thm:psg} and~\ref{thm:inv-psg} can be fused to show that this involution \ais\ has no finite identity basis whenever the group $\mathcal{G}$ is finite, solvable and non-Dedekind.
\end{remark}

\end{document}